\documentclass{aupl}

\usepackage{
amsfonts,
latexsym,
amssymb,
enumerate,
rotate,
mathrsfs,
centernot,
color,
}

\usepackage[all]{xy}

\newcommand{\labbel}[1]{\label{#1} [[{\bf #1}]]}  
\newcommand{\bibbitem}[1]{\bibitem{#1} [[{\bf #1}]]}  
\renewcommand{\labbel}{\label} \renewcommand{\bibbitem}{\bibitem}

 \newcommand{\Thetarel}{\mathrel{\Theta}}

\definecolor{bluee}{RGB}{0,0,25}

\newcommand{\arxiv}[1]{{\color{bluee}#1}}

\usepackage{stackengine,
graphicx}
\stackMath
\def\dashsubseteq{\mathrel{
  \stackinset{l}{0pt}{c}{}{\wrule[2pt]{5pt}{.5pt}}{
  \stackinset{l}{1.6pt}{c}{}{\wrule[-10pt]{.5pt}{3pt}}{
  \stackinset{l}{0.5pt}{c}{1.5pt}{\rotatebox[origin=center]{45}{\wrule[2pt]{.5pt}{3pt}}}{
  \stackinset{l}{0.5pt}{c}{}{\rotatebox[origin=center]{-45}{\wrule[-5pt]{.5pt}{3pt}}}{
  \stackinset{l}{3.35pt}{c}{}{\wrule[-4pt]{.5pt}{10pt}}{
  \stackinset{l}{5.1pt}{c}{}{\wrule[-4pt]{.5pt}{10pt}}{
 \subseteq}}}}}}
}}
\newcommand\wrule[3][0pt]{\textcolor{white}{\rule[#1]{#2}{#3}}}

\newtheorem{theorem}{Theorem}[section]
\newtheorem{lemma}[theorem]{Lemma}

\newtheorem{proposition}[theorem]{Proposition}

\newtheorem*{claim*}{Claim}

\newtheorem*{theorem*}{Theorem}
\newtheorem*{proposition*}{Proposition}
\newtheorem*{corollary*}{Corollary}
\newtheorem*{lemma*}{Lemma}
\newtheorem*{scholion*}{Scholion}

\theoremstyle{definition}
\newtheorem{definition}[theorem]{Definition}

\newtheorem{problem}[theorem]{Problem}

\theoremstyle{remark}
\newtheorem{remark}[theorem]{Remark}

\newtheorem*{remark*}{Remark}
\newtheorem*{remarks*}{Remarks}
 
\newtheorem{example}[theorem]{Example}

\newtheorem*{acknowledgement}{Acknowledgement}

\newtheorem*{observation*}{Observation}

 \allowdisplaybreaks[1]

\numberwithin{equation}{section}

\begin{document}

\title{A hyperamalgamation property}

\author{Paolo Lipparini} 
\address{Dipartimento di Hypermatematica\\Viale della  Ricerca
Scientifica\\Universit\`a di Roma ``Tor Vergata'' 
\\I-00133 ROME ITALY}

\email{lipparin@axp.mat.uniroma2.it}

\urladdr{http://www.mat.uniroma2.it/\textasciitilde lipparin}



\thanks{Work performed under the auspices of G.N.S.A.G.A. Work 
partially supported by PRIN 2012 ``Logica, Modelli e Insiemi''.
The author acknowledges the MIUR Department Project awarded to the
Department of Mathematics, University of Rome Tor Vergata, CUP
E83C18000100006.}

\begin{abstract}
The superamalgamation property is a strengthening of the amalgamation property 
and has found many applications in algebraic logic,
more recently, also in model theory.
Here we introduce an even stronger notion 
we call the \emph{hyperamalgamation property}.
We show that lattices, semilattices, Boolean algebras
and dual Heyting algebras have the hyperamalgamation property. 
As a first application, we obtain the amalgamation property
for semilattices, Boolean algebras
and Heyting algebras with  further operators.
\end{abstract}

\maketitle

\section{Introduction} \labbel{intro}

 The amalgamation property (AP) has proven to be a
very useful tool 
in algebra and logic, e.~g.,  \cite{H, KMPT,MMT}.
For recent applications in computer science see \cite{BGR}
and further references there. A stronger notion,
the superamalgamation property,  
is applicable to structures with a partial order,
and has  particular interest in 
algebraic logic \cite{GM}. 

For definiteness, here we work in the setting
of model theory, namely, a \emph{model}, or a \emph{structure} 
is, roughly, a set together with families of operations and relations.
 However, the amalgamation property is loosely related to 
the categorical notion of a push-out: if a push out exists,
then it is an amalgamating structure if and only if
 the associated morphisms are embeddings.
On one hand, there are clear
model-theoretical criteria implying that the class
of  models of some theory  has push-outs: 
this applies to every
theory axiomatizable by universal strict  Horn sentences.
On the other hand, quite surprisingly, there seems to be
no easy widely applicable condition which  implies AP 
and is formulated in terms of the form
of the axioms of some theory.
The problem has been raised 
by B. J{\'o}nsson \cite{J1} and seems to have a positive solution
only in a small handful of cases \cite{apuvar}.

Luckily, this is only half the story:
there are many cases in which 
theories with AP can be 
combined  \cite{GG,apu} 
 or
modified \cite{BHKK,GP}  in order to obtain other theories with 
AP.
In particular, in \cite{sapimpap}
we showed that if some class $\mathcal S$ of ordered structures
has  (1)  the superamalgamation property
and (2) is such that every model of $\mathcal S$ can be extended
to a model in $\mathcal S$ whose order is  a complete bounded lattice,
then the class of expansions of models of $\mathcal S$ with
an added closure operation  has still the superamalgamation property.
In the above sentence ``closure''  can be replaced by 
any one of the following:
isotone, extensive, isotone and extensive, idempotent, 
idempotent and extensive, involution, antitone.
The result applies also to $n$-ary operations, possibly
supposed to be isotone on a set of arguments
and antitone on another set of arguments.
See Definition \ref{list} for an explicit list. 

The above results are quite general: we can take as $\mathcal S$
any one of the following classes: 
partially ordered sets,  meet semilattices,  join semilattices, 
lattices,  Boolean algebras,  Heyting algebras, which all
satisfies the above conditions (1) and (2).

The case not treated in \cite{sapimpap}
is the case of an additive operation (in an
ordered structure whose order is at least a join semilattice).
Recall that a unary operation $K$  is
\emph{additive} if $K(x+y)=Kx+Ky$, for all 
$x$ and $y$ in the domain, where $+$
denotes the join operation.    
Note that we do not assume additivity in the
definition of a closure operation, unless explicitly 
otherwise stated.
 
In order to deal with additive operations, 
the superamalgamation property seems to be not enough;
however, here we introduce a stronger
``hyperamalgamation'' property which does the job
and has probably independent interest.
At a first glimpse, the hyperamalgamation property 
might bear some resemblance  
with the Maehara  Interpolation Property \cite[p. 43]{MMT} 
but here we shall  not try to pursue the analogy further.

\section{Preliminaries} \labbel{prel} 

As mentioned, we will work in the framework of
Model Theory \cite{H}. 
\emph{Structure} and  \emph{model}
will be used as synonyms.  
Formally, in Model Theory
 each structure is appropriate
for some specific language containing 
(possibly empty sets of) function, relation and
constant symbols. 
We will write $f_ {\mathbf A}$ to denote the 
interpretation of the function symbol $f$
in the structure $\mathbf A$. However,
when there is no danger of confusion,
we will simply write $f$ in place of $f _ {\mathbf A} $.
As custom, the underlying set of the structure $\mathbf A$ 
will be denoted by $A$. This will never generate ambiguities,
though sometimes two different structures $\mathbf A$ 
and  $\mathbf A^-$ will be considered on $A$.
Usually, $\mathbf A^-$ will be a \emph{reduct} of $\mathbf A$,
namely,  the language of $\mathbf A^-$ is a subset of 
the language of $\mathbf A$, with the common symbols
interpreted in the same way.
In the standard model-theoretical terminology,
in the above situation one also says  that
 $\mathbf A$ is an \emph{expansion} of $\mathbf A^-$.
Note that the notion of expansion is distinct from
the notion of \emph{extension}: in the former case
the language is expanded, but the underlying set remains the same;
in the latter case the   underlying set is extended and it is the
language which remains the same.

Semilattices will be  always considered as 
algebraic structures with a single 
binary operation $+$ and will be intended as \emph{join-semilattices},
namely,  the induced order is defined
by $a \leq b$ if $a+b=b$.  
See \cite{CHK,E,G,Ha} 
for further details about ordered sets, lattices and semilattices.

We refer to the quoted works
for other unexplained notions and notation.
Essentially we use the same conventions as in  \cite{sapimpap},
but here lattice operations
will be denoted by $+$ for join and by
$\cdot$ or juxtaposition for meet, as more usual in the
distributive case.

\section{A hyperamalgamation property} \labbel{hypsec}

Recall that a class $\mathcal S$
 of structures  of the same type,
 has the \emph{amalgamation property}
(AP) if, whenever 
$\mathbf A, \mathbf B, \mathbf C \in \mathcal S$,
 $ \iota _{\mathbf C, \mathbf A} \colon \mathbf C \rightarrowtail \mathbf A$
and 
 $ \iota _{\mathbf C, \mathbf B} \colon \mathbf C \rightarrowtail \mathbf B$
are embeddings,  there is a structure
$\mathbf D \in \mathcal S$ and  embeddings
$ \iota _{\mathbf A, \mathbf D} \colon \mathbf A \rightarrowtail \mathbf D$
and 
 $ \iota _{\mathbf B, \mathbf D} \colon \mathbf B \rightarrowtail \mathbf D$
such that 
$\iota _{\mathbf C, \mathbf A} \circ \iota _{\mathbf A, \mathbf D}
=
\iota _{\mathbf C, \mathbf B} \circ \iota _{\mathbf B, \mathbf D}$. 
 
If, in addition, 
the above model and embeddings can be
always chosen in such a way that
 the intersection
of the images of $\iota _{\mathbf A, \mathbf D}$
and $\iota _{\mathbf B, \mathbf D}$ 
is equal to the image of 
$\iota _{\mathbf C, \mathbf A} \circ \iota _{\mathbf A, \mathbf D}$, then
$\mathcal S$ is said to have the \emph{strong amalgamation property}
(SAP).

\begin{remark} \labbel{cuia}    
If $\mathcal S$ is closed under isomorphism, then, by taking 
isomorphic copies, in the  definition of SAP 
(but not of AP)
we can   reduce ourselves to the case 
when embeddings are inclusions. In this case, the above condition
for SAP
simplifies.

A class $\mathcal S$  closed under isomorphism has
the strong amalgamation property if and only if,
for all $\mathbf A, \mathbf B, \mathbf C \in \mathcal S$
such that 
 $ \mathbf C \subseteq  \mathbf A$,
 $ \mathbf C \subseteq  \mathbf B$
and $C=A \cap B$,
 there is a structure
$\mathbf D \in \mathcal S$ such that 
$ \mathbf A \subseteq  \mathbf D$ and
$ \mathbf B \subseteq  \mathbf D$.
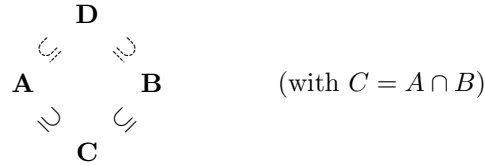
\begin{equation*}
\phantom{\qquad \qquad  \text{ (with $C=A \cap B$)}  }
 \begin{matrix} 
 \mathbf D \cr
$\rotatebox[origin=c]{45}{$ \dashsubseteq $}$
 \ \quad\  
$\reflectbox {\rotatebox[origin=c]{45}{$ \dashsubseteq $}}$ \cr
\mathbf A \quad \quad \quad \quad \mathbf B \cr
$\rotatebox[origin=c]{-45}{$ \supseteq $}$ \ \quad\ 
$\rotatebox[origin=c]{45}{$ \subseteq $}$
 \cr
 \mathbf C
 \end{matrix}   
\qquad \qquad  \text{ (with $C=A \cap B$)}  
\end{equation*}

We now describe further variations of AP.
For the sake of simplicity, we will always work in the context 
described in the present remark, so classes of similar
structures will be always supposed to be closed under isomorphism.
\end{remark} 

A stronger property has proved very useful in algebraic logic \cite{GM,KH}.
If $\mathcal S$ is a class of ordered structures, 
then $\mathcal S$ has the \emph{superamalgamation property}
(superAP) if, under the assumptions in the above remark, 
there exists some $\mathbf D$ as above 
with the additional property that, 
for every  $a \in A $ and $b \in B $, 
  \begin{enumerate}[(a)]
\item 
if $a \leq _{\mathbf D}  b $,  
then there is $c \in C$ such that
  $a \leq_{\mathbf A} c $ and
$ c \leq _{\mathbf B} b $, and, symmetrically,
\item 
if $b \leq _{\mathbf D}  a $,  
then there is $c \in C$ such that
  $b \leq_{\mathbf B} c $ and
$ c \leq _{\mathbf A} a $.
  \end{enumerate}

Note that superAP implies SAP.
In \cite{sapimpap} we have showed that superamalgamation 
is useful also from a model theoretical point of view.
Here we introduce a stronger property
which shall be needed in order to deal with
additive operations.

\begin{definition} \labbel{hypdef1}
Throughout the paper,  $\mathcal S$
 is a class of structures  of the same type,
closed under isomorphism and
with at least a semilattice operation.
For short, we will call $\mathcal S$  a
\emph{class of semilattice structures}, and members of 
$\mathcal S$ will be 
correspondingly called \emph{semilattice structures}.
In particular, a semilattice structure might
(or might not) contain further operations, relations
and constants, besides join $+$.

We say that some class $\mathcal S$  
 of semilattice structures 
has the \emph{$1$-step hyperamalgamation property}
 if, whenever 
$\mathbf A, \mathbf B, \mathbf C \in \mathcal S$,
 $ \mathbf C \subseteq  \mathbf A$,
 $ \mathbf C \subseteq  \mathbf B$
and $C=A \cap B$,
 there is a structure
$\mathbf D \in \mathcal S$ such that 
$ \mathbf A \subseteq  \mathbf D$ and
$ \mathbf B \subseteq  \mathbf D$
(so far, we have just repeated the definition of SAP), and, moreover,
  \begin{enumerate}[(a)]
\item
if  $a, a_1 \in A $, $b_1 \in B $ and
 $a \leq _{\mathbf D} 
a_1 +_{\mathbf D} b_1 $,  
then there is $c \in C$ such that
  $a \leq_{\mathbf A} a_1 +_{\mathbf A} c $ and
$ c \leq _{\mathbf B} b_1 $; and, 
\item 
if  $a_1 \in A $, $b, b_1 \in B $ and
 $b \leq _{\mathbf D}  b_1 +_{\mathbf D} a_1 $,  
then there is $c \in C$ such that
  $b \leq_{\mathbf B} b_1 +_{\mathbf B} c $ and
$ c \leq _{\mathbf A} a_1 $.
\end{enumerate}
 \end{definition}   

In principle, our results apply to structures
with many semilattice operations.
In such a situation, we always suppose that 
some specific semilattice operation
is selected; the hyperamalgamation properties 
are always meant to refer to such an operation.

If structures in $\mathcal S$ have a $0$ 
(a minimal element, interpreted as a constant, hence preserved
by embeddings),
then the $1$-step hyperamalgamation property
  implies the superamalgamation property:
just take $a_1=0$, respectively, $b_1=0$. 
\arxiv{More generally, even without the assumption of the
existence of a $0$, any instance of the 
hyperamalgamation property implies the corresponding
instance of the superamalgamation property, provided
that every $b_1 \in B$ is $\geq$ than some  
$c^*$, for $c^* \in C$  
and symmetrically for each $a_1 \in A$. Say, in the former case, 
take $c^*$ in place of $a_1$ in (a);
then a separating element is $c^* + c$.  }

Lattices and semilattices fail to have the 
$1$-step hyperamalgamation property (see Example \ref{no1} below),
hence we need a ``many steps'' weaker version.

\begin{definition} \labbel{hypdef}    
If  $\mathcal S$
 is a class of semilattice structures,
we say that $\mathcal S$  has the \emph{hyperamalgamation property}
 if, whenever 
 $ \mathbf C \subseteq  \mathbf A$,
 $ \mathbf C \subseteq  \mathbf B$ in $\mathcal S$ 
and $C=A \cap B$, 
 there is a structure
$\mathbf D \in \mathcal S$ such that 
$ \mathbf A \subseteq  \mathbf D$,
$ \mathbf B \subseteq  \mathbf D$ and, moreover,
\begin{enumerate}
\item[(c)]
for every  $a, a_1 \in A $ and $b_1 \in B $ such that 
 $a \leq _{\mathbf D} a_1 +_{\mathbf D} b_1 $,  
 there are $n \in \mathbb N^+$  and   sequences
$a_2, \dots, a _{n} \in A$, 
 $ b_2, \dots b_n \in B$,
 $c_1, c^*_2, c_2, c^*_3, c_3, \dots, \allowbreak 
 c _{n-1}, c^*_n, c_n \in C$ such that

\begin{equation}\labbel{heq}     
\begin{aligned}  
a_2 &\leq_{\mathbf A} a_1 +_{\mathbf A}  c_1
 \text{   and  } c_1 \leq _{\mathbf B} b_1  
\qquad\qquad\qquad
\\
b_2 &\leq_{\mathbf B} c^*_2 +_{\mathbf B} b_1
 \text{   and  } c^*_2 \leq _{\mathbf A} a_2 
\\
a_3 &\leq_{\mathbf A} a_2 +_{\mathbf A} c_2 
\text{   and  } c_2 \leq _{\mathbf B} b_2  
\\
b_3 &\leq_{\mathbf B} c^*_3 +_{\mathbf B} b_2
 \text{   and  } c^*_3 \leq _{\mathbf A} a_3 
\\
& \dots 
\\
a_n &\leq_{\mathbf A} a_{n-1} +_{\mathbf A} c_{n-1}
 \text{   and  }  c_{n-1} \leq _{\mathbf B} b_{n-1}  
\\
b_n &\leq_{\mathbf B} c^*_n +_{\mathbf B} b_{n-1} 
\text{   and  } c^*_n \leq _{\mathbf A} a_n 
\\
a &\leq_{\mathbf A} a_n + c_n \text{   and  } c_n \leq _{\mathbf B} b_n  
\end{aligned} 
  \end{equation}
  \end{enumerate} 
and also the symmetrical conclusion holds 
(with possibly a distinct $n$) for every 
  $b, b_1 \in B $ and $a_1 \in A $ such that 
 $b \leq _{\mathbf D} b_1 +_{\mathbf D} a_1 $.

The definition is intended in the sense that if $n=1$
we get the $1$-step hyperamalgamation property.  
\end{definition}

Again, if semilattices in $\mathcal S$ have a $0$
interpreted as a constant, 
then hyperamalgamation implies superamalgamation.
Indeed, if $a \leq_{\mathbf D} b_1$ take $a_1=0$ in (c) above. 
Then the first line in \eqref{heq} reads 
$a_2 \leq c_1 \leq b_1  $,
thus the second lines gives $b_2 \leq b_1$
and the third line gives $a_3 \leq c_1 + c_2 \leq b_1$.
Iterating, we get $a \leq c_1 + c_2 + \dots +  c_n \leq b_1$,
thus $c_1 + c_2 + \dots + c_n$ witnesses superamalgamation.

\begin{theorem} \labbel{celhan}
  \begin{enumerate}  
  \item 
Boolean algebras have the $1$-step hyperamalgamation property.
\item
Heyting algebras satisfy the dual of
the $1$-step hyperamalgamation property.
\item
Lattices and join semilattices have the 
hyperamalgamation property.
  \end{enumerate} 
 \end{theorem}

\begin{proof}
(1)  Boolean algebras have the superamalgamation property,
as proved, for example, in \cite{GM}; 
see \cite[Theorem 2.4]{sapimpap} for details
and for an alternative proof. We  use  superamalgamation 
in order to prove the $1$-step hyperamalgamation property.
If $\mathbf A$, $\mathbf  B$ and $\mathbf  C$  
 are to be hyperamalgamated, take $\mathbf  D$ a
superamalgamating model.
Then $a \leq_{\mathbf D} a_1 +_{\mathbf D} b$ in $\mathbf  D$ 
means
$a  a_1'  \leq_{\mathbf D} b$, where $'$
denotes complementation,  hence
$a  a_1'  \leq_{\mathbf A} c \leq_{\mathbf B} b$, for some $c \in C$,
by superamalgamation.
Thus 
$a \leq_{\mathbf A} a_1 +_{\mathbf A} c$ and
$c \leq_{\mathbf b} b$.

(2) is given by a dual argument, since 
in Heyting algebras
$a \geq_{\mathbf D} a_1 b$ is equivalent to 
$ a_1 \rightarrow  a \geq_{\mathbf D} b$.
  
(3) We expect that the result is
known, in some form or another. We 
present full details since we are not aware of a reference.

We first prove the result for join semilattices.
So let $\mathbf A$, $\mathbf  B$ and $\mathbf  C$ 
be as in the hypotheses of the hyperamalgamation property.
It is no loss of generality to assume that 
$\mathbf A$, $\mathbf  B$ and $\mathbf  C$  have a minimal element
$0$, since otherwise we can add such an element.
The free product of $\mathbf A$ and $\mathbf  B$  
is $\mathbf A \times \mathbf  B$, with the canonical morphism
from $\mathbf A$ to $\mathbf A \times \mathbf  B$ sending
$a$ to $(a,0)$, and symmetrically for $\mathbf  B$.
Hence the pushout $\mathbf  D$ of  $\mathbf A$ and $\mathbf  B$
over  $\mathbf  C$  is the quotient of $\mathbf A \times \mathbf  B$
under the smallest congruence $\Theta$ identifying
$(c,0)$ and $(0,c)$, for every $c \in C$.  
One sees that $\Theta$ is the smallest symmetric and
transitive relation containing the relation $\Theta'$
such that any element of the form $(a,b+c)$
is $\Theta'$ related with  $(a+c,b)$, for $a \in A$,
$b \in B$ and $c \in C$. Just observe that $\Theta'$  
is reflexive, compatible and relates $(0,c)$ and $(c,0)$.
Let $[a,b]$ denote the $\Theta$-equivalence class of $(a,b)$. 
  
We claim that $\mathbf  D$ hyperamalgamates
$\mathbf A$ and $\mathbf  B$ over $\mathbf  C$. 
Suppose that
$a, a_1 \in A $, $b_1 \in B $ and 
 $a \leq _{\mathbf D} a_1 +_{\mathbf D} b_1 $, as given by
 the hypotheses of the  hyperamalgamation property.
By the definition of $\mathbf  D$, the above inequality
reads $[a,0] \leq_{\mathbf D} [a_1, b_1]$, namely,
$[a_1, b_1] = [a,0] + [a_1,b_1]   $ in $\mathbf  D$, i.e.,
$ [a_1, b_1] = [a+a_1,b_1] $ in $\mathbf  D$, that is,
$(a_1, b_1) \Thetarel (a+a_1,b_1) $ in   $\mathbf A \times \mathbf  B$, where
the notation $ (a_1, b_1) \Thetarel (a+a_1,b_1)$ means that $  (a_1, b_1)$
 and $(a+a_1,b_1) $ are $\Theta$-related.
By the considerations in the previous paragraph,
 this means that there exists a sequence
\begin{equation}\labbel{lungh}     
  (a_1, b_1) \Thetarel'  (a_2, b^*_1) \Thetarel' _\smallsmile 
(a^*_2, b_2) \Thetarel' (a_3, b^*_2) \Thetarel' _\smallsmile
  \dots (a+a_1,b_1)
\end{equation}
where $_\smallsmile$ denotes converse.
By definition, the relation $ (a_1, b_1) \Thetarel'  (a_2, b^*_1) $
can be written as $ (a_1, b^*_1+c_1) \Thetarel'  (a_1+c_1, b^*_1) $,
for some $c_1 \in C$ with $b_1=  b^*_1+c_1 $ and
$a_2=a_1+c_1$, in particular, $a_2 \leq a_1+c_1$
and $c_1 \leq b_1$, the first line  of \eqref{heq}. 
 Note that we also have $ b^*_1 \leq  b_1$. 
Again by definition, the relation $(a_2, b^*_1) \Thetarel' _\smallsmile 
(a^*_2, b_2)$ means $(a^*_2+c^*_2, b^*_1) \Thetarel' _\smallsmile 
(a^*_2, b^*_1 + c^*_2)$, for some $c^*_2 \in C$ and
for $a_2 = a^*_2+c^*_2$ and $b_2= b^*_1 + c^*_2$,
thus $c^*_2 \leq a_2$ and  $b_2 \leq c^*_2 + b^*_1
\leq c^*_2 + b_1$, the second line of \eqref{heq}, and also
$a^*_2 \leq a_2$.

Going on, $ (a^*_2, b_2) \Thetarel'  (a_3, b^*_2) $
means $ (a^*_2, b^*_2+c_2) \Thetarel'  (a^*_2+c_2, b^*_2) $,
for some $c_2 \in C$ with $b_2=  b^*_2+c_2 $ and
$a_3=a^*_2+c_2$, from which, as above, we get the third line
in \eqref{heq}. At the end,  we reach $c_n \leq b_n$
and  $a+ a_1 \leq a_n+c_n$, 
in particular $a  \leq a_n+c_n$, the last line of \eqref{heq}.

Having proved the result for semilattices,
hyperamalgamation for lattices is immediate from
the case for semilattices and the next lemma, since
the definition of the hyperamalgamation property involves
only $\leq$ and the semilattice operation, and since we have showed that
hyperamalgamation implies superamalgamation 
(under the assumption of the existence of a $0$, but this
can always be accomplished, if necessary, by adding some new $0$ 
to $\mathbf A$, $\mathbf  B$ and $\mathbf  C$).
\end{proof}

\begin{lemma} \labbel{lemla}
Under the conventions  in 
Remark \ref{cuia}, suppose that $\mathbf A$, $\mathbf  B$, $\mathbf  C$
are lattices and the join semilattice $\mathbf  D^-$ superamalgamates
the join semilattice reducts of  $\mathbf A$, $\mathbf  B$, $\mathbf  C$.
Without loss of generality,
assume that $\mathbf  D^-$ is generated by $A \cup B$,
as a semilattice.

Then meets in $\mathbf A$, respectively, in $\mathbf  B$ are preserved
in $\mathbf  D^-$ (as a partially ordered set).
In particular, if $\mathbf  E$ is a lattice whose
semilattice reduct extends
  $\mathbf  D^-$ preserving
existing meets, then $\mathbf  E$ superamalgamates
$\mathbf A$ and $\mathbf  B$ over $\mathbf  C$.
For example, we can take $\mathbf  E$  the Dedekind-McNeille completion 
of $\mathbf  D^-$.
 \end{lemma} 

\begin{proof}
The argument appears at the end of the proof
of \cite[Theorem 3.5]{J}. We give explicit details for the reader's convenience.
Suppose that  $a_1,a_2 \in A$ have meet $a_3$ in $\mathbf A$ 
and suppose that   $d \in D$ is such that  $d \leq_{\mathbf  D^-} a_1,a_2$.
Since  $\mathbf  D^-$ is generated by $A \cup B$
as a semilattice, $d= a +_{\mathbf  D^-} b$, for some 
not necessarily unique
$a \in A$ and $b \in B$.  
Thus $a \leq_{\mathbf  A} a_1,a_2$, hence 
$a \leq_{\mathbf  A} a_3 $, since $a_3$ is the meet of  $  a_1,a_2$
in $\mathbf A$. Moreover, $b \leq_{\mathbf  D^-} a_1$,
hence, by superamalgamation, there is $c_1 \in C$ with
$b \leq_{\mathbf  B} c_1 \leq_{\mathbf  A} a_1$;
symmetrically, there is $c_2 \in C$ with
$b \leq_{\mathbf  B} c_2 \leq_{\mathbf  A} a_2$.
Hence  $b \leq_{\mathbf  B} c_1 \cdot_{\mathbf  C} c_2
\leq_{\mathbf  A} a_1 \cdot_{\mathbf  A} a_2 = a_3$.
Since all embeddings are order-embeddings, in particular,
both $\mathbf A$ and $\mathbf  B$ order-embed into $\mathbf  D^-$,
we get $d= a +_{\mathbf  D^-} b \leq _{\mathbf  D^-} a_3 $.
This shows that $a_3$ is the meet of  
$a_1 $ and $ a_2 $ also in $\mathbf  D^-$.

Meets in $\mathbf  B$ are preserved by the symmetric
argument.
The remaining statements are straightforward.
Recall that the Dedekind-McNeille completion preserves
all existing meets (and joins).
 \end{proof}

\begin{example} \labbel{no1}   
We show that lattices and semilattices do not have
the $1$-step hyperamalgamation property.
Let $\mathbf A$ and $\mathbf  B$ be the lattices pictured below, with
the common sublattice $\mathbf  C$, the chain with elements
$0 < c_1 < c^*_2 <  c_2 < c^*_3 \dots$.
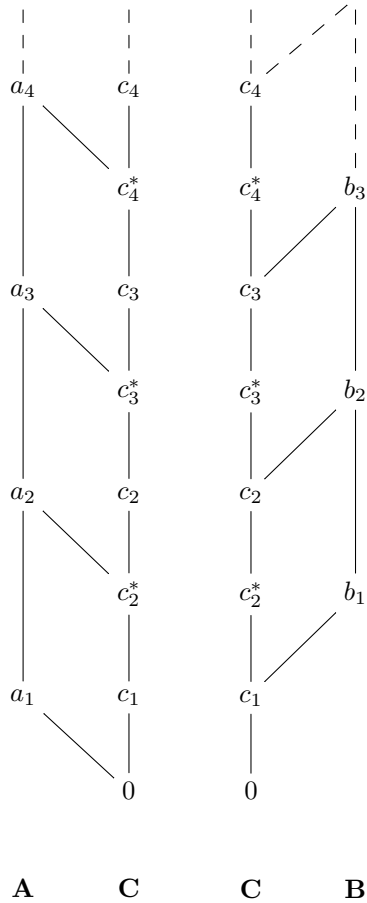
\begin{equation*}
\xymatrix{
&&
\\
a_4  \ar@{--}[u] &
c_4 \ar@{--}[u]  
\\
& c^*_4  \ar@{-}[lu] \ar@{-}[u] 
\\
a_3  \ar@{-}[uu] &
c_3 \ar@{-}[u]   
\\
& c^*_3  \ar@{-}[lu] \ar@{-}[u] 
\\
a_2  \ar@{-}[uu] &
c_2 \ar@{-}[u]  
\\ 
& c^*_2 \ar@{-}[lu] \ar@{-}[u] 
\\
a_1  \ar@{-}[uu] & 
c_1 \ar@{-}[u]  
\\ 
& 0 \ar@{-}[lu] \ar@{-}[u]   &
\\
\mathbf A & \mathbf  C
}
\xymatrix{
&&
\\
 c_4 \ar@{--}[ru] \ar@{--}[u]   
\\
 c^*_4   \ar@{-}[u]  & b_3 \ar@{--}[uu]  
\\
 c_3 \ar@{-}[ru] \ar@{-}[u]   
\\
 c^*_3   \ar@{-}[u] & b_2 \ar@{-}[uu]
\\
c_2 \ar@{-}[ru] \ar@{-}[u]   
\\ 
c^*_2  \ar@{-}[u] & b_1 \ar@{-}[uu]
\\
 c_1 \ar@{-}[ru] \ar@{-}[u]  
\\ 
 0  \ar@{-}[u]
\\
\mathbf  C& \mathbf  B
}
\end{equation*}

In any amalgamating semilattice $\mathbf  D$ we have
$a_2 = a_1+ _{\mathbf  A} c_1 \leq_{\mathbf  D} a_1+_{\mathbf  D}b_1  $,
 since $c_1 \leq_{\mathbf  B} b_1$.
Since $c^*_2 \leq_{\mathbf  A} a_2 \leq_{\mathbf  D} a_1+b_1$, 
we get    $ b_2 =  c^*_2 +_{\mathbf  B} b_1 \leq_{\mathbf  D} 
a_1+_{\mathbf  D}b_1 $.
The $1$-step hyperamalgamation property
(case (b) with $b_2$ in place of $b$)
should provide some $c \in C$ such that 
$b_2 \leq_{\mathbf  B} b_1+_{\mathbf  B}c  $
 and $c \leq_{\mathbf  A} a_1$,
but the only element of $C$ less than $a_1$ is $0$,
hence   the $1$-step hyperamalgamation property fails.

Of course, we used only the bottom part of the above lattices,
in order to provide a counterexample to  the 
$1$-step hyperamalgamation property.
But the example shows, more generally, that there is
no limit to the length of the sequence in \eqref{heq}.
Indeed, iterating the argument,
 from $c_2 \leq_{\mathbf  B} b_2$ we get 
    $ a_3 = a_2 +_{\mathbf  A} c_2 \leq_{\mathbf  D} 
a_1+_{\mathbf  D}b_1 $.
 Iterating again, we get that  in any amalgamating semilattice
$a_1+_{\mathbf  D}b_1$ 
 is greater than all the elements pictured in the diagram.
\end{example}

\section{Preservation of 
hyperamalgamation in expansions} \labbel{hexpa}

\begin{definition} \labbel{operat}
Assume that $\mathbf E$ is a semilattice structure
and $K: E \to E$ is a unary operation
(either for the language of
$\mathbf E$, or not in the language of $\mathbf E$).

We will consider the
following properties
of $K$.  
  \begin{enumerate}[({D}1)]   
 \item 
$K$ is \emph{additive}, that is,   $K(a+b)=Ka+Kb$, for all
$a, b \in E$,
\item
$K$ is additive and extensive,
where \emph{extensive} means that  $Ka  \geq a$, for all
$a \in E$, and
\item
$K$ is additive and idempotent,
where \emph{idempotent} means that  $KKa = Ka$, for all
$a \in E$, and
\item
$K$ is
an \emph{additive closure operation}, that is, $K$ 
is additive, extensive and idempotent. 
  \end{enumerate} 

Some authors include additivity in the definition of a closure
operation, but, following standard use
in order theory, we will always make explicit mention 
of additivity. Note that an additive operation is necessarily
\emph{isotone}, namely, $a \leq b$ implies $Ka \leq Kb$.  

Now assume further that  $D \subseteq E$
 and $G: D \to E$ is a function, in other words,
 $G$ is a \emph{partial} 
function on $E$.
We will consider the following properties
of $G$.  
  \begin{enumerate}[({D*}1)] 
  \item
$G$ is \emph{suitable for additivity}
 if 
\begin{equation}\labbel{isoa}\tag{Eq. D*1}
d \leq d_1 + \dots + d_n  \text{ implies } Gd \leq Gd_1 + \dots + Gd_n
\end{equation}
for every
$n \in \mathbb N^+$  and  all $d,d_1, \dots, d_n \in D$, 
   \item   
$G$  is  \emph{suitable for additivity and extensiveness} if 
(D*1) above holds and furthermore $Gd \geq d$, for all
$d \in D$,
\item 
$G$ is \emph{suitable for additivity and idempotency} 
if,  for every
$i, n \in \mathbb N^+$ with $i \leq n$,  and  all $d,d_1, \dots, d_n \in D$,

\begin{align}
\labbel{isoidea}\tag{Eq. D*3}
&\begin{aligned}    
&d \leq d_1 + \dots + d_i + Gd_{i+1}+ \dots + Gd_n 
 \text{ implies }
\\ 
& \qquad \qquad \qquad Gd \leq Gd_1 + \dots +
 Gd_i + Gd_{i+1}+ \dots Gd_n \text{ and }
 \end{aligned}
\\
 \labbel{isoideabis}\tag{Eq. D*3$'$}
&\begin{aligned}    
&Gd \leq d_1 + \dots + d_i + Gd_{i+1}+ \dots + Gd_n 
 \text{ implies }
\\ 
& \qquad \qquad \qquad Gd \leq Gd_1 + \dots +
+ Gd_i + Gd_{i+1}+ \dots Gd_n  
 \end{aligned}
\end{align} 
\item
 $G$ is  \emph{suitable for additive closure}   if 
$G$ is extensive and
\begin{equation}\labbel{cloa}\tag{Eq. D*4}
d \leq Gd_1 + \dots + Gd_n  \text{ implies } Gd \leq Gd_1 + \dots + Gd_n
\end{equation} 
for every
$n \in \mathbb N^+$  and  all $d,d_1, \dots, d_n \in D$.
\end{enumerate} 
 \end{definition}   

Note that if $G$ is total on $E$, that is, $D=E$,
then, for each $n=1, \dots, 4$, Condition (D*n)
is equivalent to  Condition (Dn), taking
$K=G$, of course. 

The next proposition  shows that ``suitability''  
is preserved by hyperamalgamating expansions,
taking $D= A \cup B$. It follows that  if, in  some class $\mathcal S$ 
with hyperamalgamation,  ``suitability'' really provides
 the possibility of extending $G$ to a total operator $K$
satisfying the corresponding properties, then hyperamalgamation 
is preserved when expanding structures in $\mathcal S$ 
by adding an operator satisfying such properties.

\begin{proposition} \labbel{suithyp} 
If  in the assumptions of the hyperamalgamation property 
 $\mathbf A$, $\mathbf  B$ and  $\mathbf  C$ 
are semilattice structures with some unary operation $K$,
let $\mathbf A^-$, $\mathbf  B^-$, $\mathbf  C^-$ 
be the corresponding reducts obtained by removing $K$.

 Suppose that  $\mathbf E^-$ hyperamalgamates 
$\mathbf A^-$ and  $\mathbf  B^-$ over $\mathbf  C^-$,
let $D=A \cup B$ and define $G: D \to E$ by    
\begin{equation}\labbel{kd}
Gd =\begin{cases} 
K_{\mathbf A} d &    \text{if  $ d \in A  $},
\\ 
K_{\mathbf B} d &    \text{if  $ d \in B  $}
 \end{cases} 
\end{equation}   
(note that $K_{\mathbf A} $ and 
$K_{\mathbf B} $ agree on $C=A \cap B$, 
 by the assumptions in the hypothesis of 
hyperamalgamation).

Then we have: for each $n=1, \dots, 4$, if $K$ 
satisfies  Condition (Dn) in $\mathbf A$, $\mathbf  B$, $\mathbf  C$,
then $G$ satisfies Condition (D*n).
\end{proposition} 

 \begin{proof} 
First consider the case $n=1$.
Suppose that, say, $d \in A$, the case 
$d \in B$  is treated in the symmetrical way.

Each $d_1$, \dots,  $d_n$
 belongs to $A$ or to $B$ (possibly to both) so,  by rearranging
the elements, we may assume that  
$d_1, \dots,  d_j \in A$ and 
$d_{j+1}, \dots,  d_n \in B$.
Letting $a_1= d_1 + _ {\mathbf A} \dots + _ {\mathbf A}  d_j$ 
and
 $b_1= d_{j+1} + _ {\mathbf B} \dots + _ {\mathbf B}  d_n$,
we have 
$d \leq a_1 + _ {\mathbf A}   b_1 $,
 Since $\mathbf E$ is hyperamalgamating,
we have elements as in Clause (c) in Definition \ref{hypdef}
(here $d$ is in place of  $a$). 
By additivity of $K _ {\mathbf A} $ and isotony 
of  $K _ {\mathbf B} $ we get 
$K _ {\mathbf A}  a_2 \leq_{\mathbf A} 
K _ {\mathbf A}  a_1 +_{\mathbf A} K _ {\mathbf A}   c_1 $
 and   $   K _ {\mathbf B} c_1 \leq _{\mathbf B} K _ {\mathbf B} b_1  $ 
from the first line in \eqref{heq}. 
Going on this way, we reach
$K _ {\mathbf A}  d \leq_{\mathbf A} 
K _ {\mathbf A}  a_n +_{\mathbf A} K _ {\mathbf A}   c_n $
and  $   K _ {\mathbf B} c_n \leq _{\mathbf B} K _ {\mathbf B} b_n  $
(recall that here we have $d$ in place of $a$).

Now, since $K _ {\mathbf A} c = K _ {\mathbf C} c =
K _ {\mathbf B} c$, for every $c \in C$,
 working in $\mathbf E$,  we can turn back getting 
$K _ {\mathbf A}  d \leq 
K _ {\mathbf A}  a_n + K _ {\mathbf C}   c_n  \leq 
K _ {\mathbf A}  a_n + K _ {\mathbf B}   b_n$.
Then, by the proved analogue of the penultimate line in \eqref{heq}, 
 $K _ {\mathbf B} b_n \leq K _ {\mathbf C} c^*_n +K _ {\mathbf B} b_{n-1} 
\leq  K _{\mathbf A} a_n +K _ {\mathbf B} b_{n-1}$, hence
 $K _ {\mathbf A} d \leq 
 K _{\mathbf A} a_n +K _ {\mathbf B} b_{n-1}$.
Going on in the same way, we reach
 $K _ {\mathbf A} d \leq 
  K _{\mathbf A} a_1 +K _ {\mathbf B} b_{1}$.
Since $K _{\mathbf A}$ and $K _{\mathbf B}$ 
are additive, recalling the definitions of $a_1$ and $b_1$,
we have    
 $K _ {\mathbf A} d \leq 
  K _{\mathbf A} d_1 + \dots + K _{\mathbf A} d_j + 
K _ {\mathbf B} d_{j+1} + \dots + 
K _ {\mathbf B} d_{n}$, namely, by the definition of $G$,
 $G d \leq G d_1 + \dots +  Gd_n$, what we need to show
in  (D*1).

The case $n=2$ is immediate from the case $n=1$, since if 
$K$ is extensive, then $G$ is extensive, as well.

Since  both in  case $n=3$  and in case $n=4$  $K$ is assumed to be
additive, we can use the already proved case $n=1$.
Note also that, under the assumptions in the present
proposition, in such cases,
$D$ is closed under applications of $G$, namely,
$G(d) \in D$, for $d \in D$, because of \eqref{kd}
and  since both $ A$ and $B$
are closed under $K$. 
 So, for example,
to get \eqref{isoideabis}, from 
$Gd \leq d_1 + \dots + d_i + Gd_{i+1}+ \dots + Gd_n $
we get
$GGd \leq Gd_1 + \dots + Gd_i + GGd_{i+1}+ \dots + GGd_n $
from the already proved  case $n=1$.
This means 
$ Gd \leq Gd_1 + \dots +
 Gd_i + Gd_{i+1}+ \dots Gd_n  $,
since, by the assumptions in (D3),
$KK d = Kd$, for $d \in A$  or $d \in B$,
hence also  $GG d =Gd$,
for $d \in D$. 
The other conditions are dealt with 
in a similar way.
\end{proof}

\begin{definition} \labbel{suit} 
Suppose that $\mathcal S$ is a class
of semilattice structures in a language which
contains a unary 
operation symbol $K$ and let $\mathcal S^-$
be the class of reducts of members of $\mathcal S$
obtained by removing $K$.
Fix some $n \in \{ \, 1,2,3,4  \,\}$.

We say that \emph{suitability is effective} for Clause (Dn)
(\emph{relative to} $\mathcal S$ and $\mathcal S'$, when they are 
 not clear from the context) if, for every model $\mathbf E^-$ 
of  $\mathcal S^-$,
every $D \subseteq E$ and every function $G:D \to E$
satisfying Clause (D*n), the model $\mathbf E^-$ can be extended
and then expanded to some model $ \mathbf F$ of $\mathcal S$ 
in such a way that the restriction of $K _ {\mathbf F} $ to $D$ 
is $G$. 
\end{definition}

\begin{theorem} \labbel{pres} 
Under the assumptions in the above definition, 
suppose that 
   \begin{enumerate}  
  \item  
  $\mathcal S^-$ has
the hyperamalgamation property and   
\item
suitability is effective for Clause (Dn).  
 \end{enumerate} 

Then $\mathcal S$ has the hyperamalgamation property,
in particular, the superamalgamation property. 
\end{theorem} 

\begin{proof}
Suppose that $\mathbf A$, $\mathbf  B$, $\mathbf  C$
are structures in $\mathcal S$ satisfying the
assumptions of the hyperamalgamation property
and let $\mathbf A^-$, $\mathbf  B^-$, $\mathbf  C^-$
be the corresponding reducts in the language without $K$.  
The hyperamalgamation property for $\mathcal S^-$ 
provides some hyperamalgamating model $\mathbf E^-$.  
If  $D$ and $G$ are defined as in 
Proposition \ref{suithyp}, then the proposition 
implies that $G$ satisfies Condition (D*n). 
By effectiveness of suitability as in Definition \ref{suit}, we get an
extension and expansion $ \mathbf F$ of $\mathbf E^-$
with an operation $K _ {\mathbf F} $
satisfying (Dn)
and extending $G$. In particular, $ \mathbf F \in \mathcal S$.  

Then $ \mathbf F$ hyperamalgamates $\mathbf A$ and $\mathbf  B$
over $\mathbf  C$. Indeed, $ \mathbf F^-$ extends $\mathbf E^-$,
thus $ \mathbf F^-$ still hyperamalgamates $\mathbf A^-$ and $\mathbf  B^-$
over $\mathbf  C^-$, since $\mathbf E^-$ does it.
Moreover, the restrictions of $K _ {\mathbf F} $
to $A$ and $B$ agree with $K _ {\mathbf A} $
and $K _ {\mathbf B} $, because of the definition of $G$
and since $K _ {\mathbf F} $ extends $G$.
 \end{proof}

In view of Theorem \ref{pres},
it is useful to determine conditions under which 
suitability is effective, which is our objective
in the next section.

\section{An extension lemma} \labbel{extsec}

The following lemma generalizes various
results from \cite{E,ecca,sapimpap,MT,Se,Sc}.
As in \cite[Lemma 3.1]{sapimpap},
extensiveness is not necessarily assumed
and  no assumption is made on the subset $D$,
in particular, $D$ 
is  not required, say, to be closed with respect to joins.
This seems
to be the main advantage with respect to previous results.
On the other hand, here working with a complete bounded lattice 
as in \cite{sapimpap} is not enough and  
we need a rather strong form of infinitary distributivity.

Recall that a complete lattice $\mathbf L$  satisfies
the \emph{Meet
Infinite Distributive Identity} (MID)
if $x + \prod _{y \in Y} y= \prod _{y \in Y} (x + y) $,
for every $x \in L$ and $Y \subseteq L$.
The property (JID) is the dual condition.
Of course, from (MID) we obtain the following
\emph{$(2, \infty )$-distributive law}
$\prod _{x \in X} x + \prod _{y \in Y} y
= \prod _{x \in X, y \in Y} (x + y) $.

\begin{lemma} \labbel{lem}
Assume that $\mathbf L$ is a complete bounded lattice
satisfying (MID),
 $D \subseteq L$ and $G: D \to L$ is a function.

Then, for every $n=1, \dots, 4$,  
 $G$ can be extended to an operation $K$ on $\mathbf L$
satisfying  (Dn) in Definition \ref{operat} if and only if
$G$ satisfies
Clause (D*n).
 \end{lemma}

\begin{proof}
Necessity is obvious in each case and no special assumption
on $\mathbf L$ is needed; actually,  $\mathbf L$ might be just 
a join semilattice. Recall that if $K$ is additive, then $K$ 
is isotone.
 
We now prove  the nontrivial directions.

(1) 
Given $G$ as in the assumptions,
let
\begin{equation} \tag{Case 1} 
\labbel{kisoa}
  Kx= \prod \{ \,  Gb_1 + \dots + Gb_n  \mid 
n \in \mathbb N^+, b_1, \dots, b_n\in D, 
x \leq b_1 + \dots + b_n  \,\}
 \end{equation} 

With the above definition $K$ is obviously isotone, hence
$K(x+y) \geq Kx$ and  $K(x+y) \geq Ky$,
thus $K(x+y) \geq Kx + Ky$. Conversely,
\begin{equation}\labbel{kkkkk4} 
\begin{aligned} 
  Kx + Ky &=
 \prod \{ \,  Gb_1 + \dots + Gb_n  \mid  b_1, \dots, b_n\in D, 
x \leq b_1 + \dots + b_n  \,\} 
\\
&+ \prod \{ \,  Gc_1 + \dots + Gc_m  \mid  c_1, \dots, c_m \in D, 
y \leq c_1 + \dots + c_m  \,\}
\\
&= ^{\text{$(2, \infty )$-dist}} 
 \prod \{ \,  Gb_1 + \dots + Gb_n + 
 Gc_1 + \dots + Gc_m \mid 
\\
& \qquad \qquad  b_1, \dots, c_m \in D, 
x \leq b_1 {+} \dots {+} b_n, y \leq c_1 {+} \dots {+} c_m   \,\}
\\
& \geq 
 \prod \{ \,  Gb_1 + \dots + Gb_n + 
 Gc_1 + \dots + Gc_m \mid 
\\
& \qquad \qquad  b_1, \dots, c_m \in D, 
x + y \leq b_1 {+} \dots {+} b_n{+} c_1 {+} \dots {+} c_m   \,\} 
\\
&= K(x+y)
 \end{aligned}
  \end{equation}     
where in
the middle identity
 we have used the infinitary distributivity assumption about $\mathbf L$. 

In the remaining parts of the proof
we will tacitly assume that $b_1, \dots, b_n\in D$. 
The assumption (D*1) shall be used in order to show that
$K$ extends $G$. In fact, if 
$x \in D$ and   
$x \leq b_1 + \dots + b_n$, then  
$Gx \leq Gb_1 + \dots + Gb_n$
by \eqref{isoa}, thus $Gx \leq Kx$.
On the other hand, 
if $x \in D$, we can take $n=1$
and $b_1=x$ 
in the defining set for $K$,
getting $Kx \leq Gx$.     

(2) If $G$
is extensive, then $K$ turns out to be extensive, as well, so (2)
follows from (1).

(3)
We need a preliminary definition.
 Let
\begin{multline} 
\labbel{kaddide1}
 Hx= \prod \{ \,  Gb_1 + \dots + Gb_i +Gb_{i+1}+ \dots + Gb_n  \mid  i \in \mathbb N, n \in \mathbb N^+, i \leq n,
\\
b_1, \dots, b_n\in D, 
x \leq b_1 + \dots + b_i +Gb_{i+1}+ \dots+ Gb_n  \,\},
 \end{multline} 
for $x \in L$.

We first observe that 
$HHx \leq Hx$, for every $x \in L$.
Indeed, if $x \leq b_1 + \dots + b_i +Gb_{i+1}+ \dots+ Gb_n$,
then, by the very definition of $Hx$,
we have  $Hx \leq Gb_1 + \dots + Gb_i +Gb_{i+1}+ \dots+ Gb_n$,
hence 
\begin{align}
&\{ \,  Gb_1 + \dots + Gb_n  \mid  
x \leq b_1 + \dots + b_i +Gb_{i+1}+ \dots+ Gb_n  \,\}
\subseteq 
\nonumber
\\
\labbel{bibba}
  &\{ \,  Gb_1 + \dots + Gb_n  \mid  
Hx \leq Gb_1 + \dots + Gb_i +Gb_{i+1}+ \dots+ Gb_n  \,\},
  \end{align} 
which implies $HHx \leq Hx$,
since all the sums $Gb_1 + \dots + Gb_n$
in \eqref{bibba}  belong to the set whose product defines 
$HHx$, taking $Hx$
in place of $x$ and 
 $i=0$ in \eqref{kaddide1}.   
 
Next, we show that if $x \in D$,
then also $Gx=Hx \leq HHx$, thus 
if $x \in D$, then $Gx=Hx = HHx$
by the above paragraph.   
First, arguing as in case (1) and using
\eqref{isoidea}, we get that if $x \in D$,
then $H x = Gx$. 
Moreover,
 if $x \in D$  and
 $Hx  \leq b_1 + \dots + b_i +Gb_{i+1}+ \dots+ Gb_n$,
then
\begin{equation}\labbel{purippa}      
Hx = Gx \leq Gb_1 + \dots + Gb_i +Gb_{i+1}+ \dots+ Gb_n
 \end{equation}
by \eqref{isoideabis}, hence
$Hx \leq HHx$, since 
$HHx$ is the product of 
all the sums in the right-hand side
of \eqref{purippa}. Thus $Hx = Gx $, for  $x \in D$.

However, we have no guarantee that 
$HHx=Hx$ holds for \emph{every} $x \in L$,
hence   a more involved definition is necessary. 
For $\alpha$ an ordinal, let
\begin{equation}\labbel{ord}      
\begin{aligned} 
K^0x&=x,
\\
  K^ {\alpha +1} x &= H K^ \alpha x \quad \text{ and } 
\\ 
 K^ \beta x &= \prod _{ \alpha  < \beta } K ^\alpha x \quad \text{ for $\beta$ limit.}
\end{aligned}
 \end{equation}
Then, for every $x \in L$,  pick 
the smallest ordinal $\alpha$ such that 
$K^ {\alpha +1} x=  K^ \alpha x$
and  set 
\begin{align}   
\labbel{kkkkk} \tag{Case 3}
&K x= K ^ \alpha x.
 \end{align}

Such an $\alpha$ exists, since 
$HHx \leq Hx$, for every $x \in L$, hence
it follows that the sequence 
$(K ^ \delta   x)$, $ \delta  $ an ordinal $> 0$,
is decreasing (note that 
$H$ is obviously isotone. This is needed in order
to show $H K ^ \beta x \leq K ^ \beta x$, for $\beta$ limit.)
Moreover, by transfinite induction on $\beta \geq \alpha $ 
\begin{equation}\labbel{dudda}
\text{if $x \in L$ and $K^ {\alpha +1} x=  K^ \alpha x$,
for some $\alpha$,
then $K^ {\beta} x =  K^ \alpha x$,
for every $\beta \geq \alpha$.} 
   \end{equation}    

Since we have showed that 
$HHx=Gx=Hx$, for $x \in D$, then
$K^2x = HK^1x=HHx=Gx=Hx=K^1x$,
for $x \in D$, hence $K$ extends $G$.
  
It is elementary to see  that  $K$ is  idempotent.
If we prove that 
$K^ \alpha (x + y) = K^ \alpha x + K^ \alpha y $,
for every ordinal $\alpha$ and all $x,y \in L$,
 we get that $K$ is additive, by \eqref{dudda}, taking $\alpha$ 
sufficiently large.

The proof is by transfinite induction on $\alpha$.
The result is obvious for $\alpha=0$.
Arguing as in \eqref{kkkkk4} 
and using $(2, \infty )$-distributivity, we get 
$H(x+y) = Hx + Hy$, for every 
$x,y \in L$, that is, the case
 $\alpha=1$. We now prove the remaining cases.

(Successor step) If $K^ \alpha (x + y) = K^ \alpha x + K^ \alpha y $
holds for some $\alpha \geq 1$ and all $x,y \in L$,
then 
\begin{align*} 
 K^ {\alpha+1} (x + y) & =
HK^ \alpha (x + y) = H(K^ \alpha x + K^ \alpha y )
\\
& =
 HK^ \alpha x + HK^ \alpha y =
K^ {\alpha+1}x+K^ {\alpha+1}y,
 \end{align*}
where we have used the inductive hypothesis, and then
the (already proved) case $ \alpha =1$ with  
$K^ \alpha x + K^ \alpha y $ in place of $x+y$.  

(Limit step)
  Suppose that $\beta$ is a limit ordinal and that 
$K^ \alpha (x + y) = K^ \alpha x + K^ \alpha y $
holds for every $\alpha < \beta $ and all $x,y \in L$. Then
\begin{align*} 
 K^ \beta (x+y) &= \prod _{ \alpha  < \beta } K ^\alpha (x+y)
=\prod _{ \alpha  < \beta } (K ^\alpha x+ K ^ \alpha y)
 =^* \prod _{ \delta , \gamma  < \beta } (K ^ \delta  x+ K ^ \gamma  y)
\\
&= ^{\text{$(2, \infty )$-dist}} 
\prod _{ \delta   < \beta } K ^ \delta   x + \prod _{ \delta   < \beta } K ^ \delta  y
= K^ \beta (x) + K^ \beta (y)
 \end{align*} 
where the equality marked by $^*$ follows from
the already proved fact that if
$\alpha \geq \delta $ and $ \alpha \geq \gamma  $,
then  $K ^\alpha x \leq K ^ \delta  x$
and $K ^\alpha y \leq K ^ \gamma   y$,
thus $K ^\alpha x +  K ^\alpha y
\leq K ^ \delta  x + K ^ \gamma   y$, hence
$\prod _{ \alpha  < \beta } (K ^\alpha x+ K ^ \alpha y)
\leq K ^ \delta  x+ K ^ \gamma  y$.
Since this is verified 
for every   $\delta, \gamma < \beta $, we  
get  $\prod _{ \alpha  < \beta } (K ^\alpha x+ K ^ \alpha y)
\leq \prod _{ \delta , \gamma  < \beta } (K ^ \delta  x+ K ^ \gamma  y)$,
 the reverse inequality being obvious.

(4) follows immediately from (3), since if $G$ is extensive,
then \eqref{cloa} is equivalent to the conjunction of
\eqref{isoidea} and  \eqref{isoideabis}.
Indeed,  \eqref{cloa} is the special case $i=0$
of  \eqref{isoidea}. Conversely, 
if $G$ is extensive, then
\eqref{cloa} implies  \eqref{isoidea}.
Condition  \eqref{isoideabis} is trivially satisfied 
when $G$ is extensive. Finally, notice that if 
$G$ is extensive, then the definition in \eqref{kkkkk} 
provides an extensive operation
$K$. 

However, a direct proof of (4) is much easier, just let
\begin{equation} \tag{Case 4} 
\labbel{kcloa}
  Kx= \prod \{ \,  Gb_1 + \dots + Gb_n  \mid 
n \in \mathbb N^+, b_1, \dots, b_n\in D, 
x \leq Gb_1 + \dots + Gb_n  \,\}.
 \end{equation} 
Arguing as above and using (D*4)
we get that $K$ extends $G$. 
Moreover, $K$ is obviously extensive, in particular,
$KKx \geq Kx$, while the same argument we have used
in order to prove that $HHx \leq Hx$ in (3)
shows that  $KKx \leq Kx$, for every $x \in L$,
hence $K$ is idempotent. 
The argument in (1), using $(2, \infty )$-distributivity,
shows that $K$ is additive.
\end{proof}

\arxiv{Of course, all the dual results hold,
for example, under the assumptions in Lemma \ref{lem}, 
 but assuming (JID) in place of (MID),
 $G$ can be extended to an operation $K$ on $\mathbf L$
such that
$K$ is a multiplicative interior operation   if and only if 
$G$ is contractive and,  for every
$n \in \mathbb N^+$  and  all $a_1, \dots, a_n, b \in D$,
\begin{equation}\labbel{intea}
 Ga_1 \cdot Ga_2 \cdots   Ga_n \leq b   \text{ implies }
 Ga_1 \cdot Ga_2 \cdots   Ga_n \leq Gb.
\end{equation} }

\section{Applications} \labbel{presap}

\begin{theorem} \labbel{app}
The classes of semilattices, Boolean algebras
and dual Heyting algebras with an additive 
(respectively, additive and extensive, 
additive and idempotent, additive closure)
unary operation all have the hyperamalgamation property,
in particular, the amalgamation property.
  \end{theorem}

 \begin{proof}
Known results \cite[Theorems 9.9, 9.10]{Ha},  \cite[Theorem 2.1]{Ko},
\cite[Theorem 2.39]{GM}
  show that each algebra in each class
can be extended to an algebra with a complete   distributive
lattice. Complete Heyting algebras
satisfy (JID) \cite[p. 174]{BD}, so that dual Heiting 
algebras satisfy (MID). In the remaining
cases, the quoted results
provide embeddings into a power set algebra, which is 
completely distributive in view of \cite[Theorem 14.5]{Ko}.

If $\mathbf S$ is, say, the class of Boolean algebras 
with an  operation satisfying (Dn) and $\mathbf S^-$ is the class
 of Boolean algebras, then by the above paragraph and
Lemma \ref{lem}, suitability is effective for clause (Dn).
By Theorem \ref{celhan} the classes under consideration
have the hyperamalgamation property, hence the conclusion 
follows from Theorem \ref{pres}.
\end{proof}  

We have stated Theorem \ref{app}
in terms of dual Heyting algebras for uniformity;
of course, the theorem holds for Heyting algebras
in the standard sense when dealing with multiplicative
``interior'' operations.

\begin{definition} \labbel{list} 
The results of the present paper can be easily merged 
with results from \cite{sapimpap}. There
we have considered the following properties
of a unary operation 
in an ordered  structure:   
   (A1e) extensive, 
   (A1c) contractive,
   (A2) idempotent,
   (A2e) idempotent and extensive, 
   (A2c) idempotent and  contractive,
   (A3) involutive,
   (B1)  isotone, 
    (B1e)  isotone and extensive, 
   (B1c)  isotone and contractive,
   (B2) isotone and idempotent,
   (B3) a closure operation (that is, 
 isotone, extensive and idempotent),
   (B4) an interior operation (that is, 
 isotone, contractive and idempotent), 
   (B5) an antitone operation,
and, for $n$-ary operations,
 (C1)$_{i,j,n}$  
isotone on the first $i$ components and
antitone on the last $j$ components,
and, possibly, 
 (C2)$_{i,j,n,h}$  
 satisfying also   $x_1 \leq F\bar x, \dots, 
x_h \leq F\bar x$, for some $h \leq i$,
   more generally,
 for lattice-ordered structures,    
 (C3)$_{i,j,n,t}$ 
 satisfying    $t(x_1, \dots, x_i) \leq F\bar x$,
for some given lattice term $t$.
\end{definition}   

\begin{theorem} \labbel{appb}
The classes of semilattices, Boolean algebras
and dual Heyting algebras with an  operation 
satisfying one of the conditions (A1e)-(D4)
from Definitions \ref{list} and  \ref{operat}  
all have the hyperamalgamation property,
in particular, the amalgamation property.
 
More generally, the superamalgamation property 
is preserved by adding any family of such operations, possibly 
satisfying distinct properties.
 \end{theorem}

\begin{proof}
The arguments in \cite{sapimpap}
deal with the superamalgamation property but,
once we know that   semilattices, Boolean algebras
and dual Heyting algebras have the hyperamalgamation property,
we can repeat all the arguments in \cite{sapimpap}
by using a hyperamalgamating structure.
Note that, in each case, the structures we have constructed
in the proof of Theorem \ref{celhan} have also the
superamalgamation property, even when a constant $0$
is not present, so the arguments from \cite{sapimpap}
carry over.   

As for the second statement, as in \cite{sapimpap},
just observe that the construction of, say, $\mathbf E^-$
does not depend on the operation to be added.
Moreover, given a set of operations, each operation
has no influence on the other operations, hence the construction
can be carried over in a uniform way. 
\end{proof}
 
\begin{remark} \labbel{genboh}
   Results analogue to Theorems \ref{app} and \ref{appb} 
probably apply to many more classes of structures. 
By Theorem \ref{pres}, all we needed in the proof of Theorem \ref{app}
   are the hyperamalgamation property and
effective suitability (Definition \ref{operat}), properties
which look quite natural and relatively weak,
though, in order to prove the latter property, we used the
rather strong requirement of the infinitary distributivity
property (MID). 
However, (MID) has been used only in Section \ref{extsec},
hence it is hoped that the methods in the present paper
can be applied to many more theories, once one finds
suitable extension lemmas parallel to \ref{lem}.
In this connections, methods and results from  
various papers by M. Gehrke
and coauthors might be relevant,
e.~g., \cite{GH}.
 \end{remark}

Finally, we present a decidability result
completely analogue to \cite{sapimpap}
and which relies only on  Section \ref{extsec},
with no use of amalgamation properties. 

\begin{theorem} \labbel{decid}
Suppose that $T$ is a locally finite universal theory 
of semilattice structures and suppose that 
 every finite model of $T$ 
 can be extended to a finite 
distributive lattice-ordered model of $T$. 

Suppose that (W) is any one of the properties
(A1e) -  (D4)\footnote{In case (C3), for consistency,
we assume that $T$ contains the theory of distributive lattices.}  
 listed in  Definitions \ref{list}  and  \ref{operat}.
Let $\mathscr L'=\mathscr L \cup \{ K \} $,
where $K$ is a new  operation symbol of corresponding arity.
 Let $T^{\text{(W)}}$ in the language $\mathscr L'$
be the extension of $T$  obtained by
adding axioms   asserting that $K$ satisfies (W). Then
the following hold. 
  \begin{enumerate}   
 \item 
If $\varphi$  is a  universal  
 sentence in $\mathscr L'$
and $\varphi$  fails in some model of  $T^{\text{(W)}}$,
then $\varphi$  fails in some finite model of  $T^{\text{(W)}}$.
\item
Suppose  that $T$ has an effective
finite axiomatization, $\mathscr L'$ is finite and  
 there is an 
effectively computable  function 
$h_{_T}$ such that, for every $n \in \mathbb N$,
every model of $T$ generated by  $ n$ elements  
can be extended to a distributive lattice-ordered model of $T$
of cardinality $\leq h_{_T}(n)$.
Then the set of all the 
universal 
 consequences
of  $T^{\text{(W)}}$ is decidable.
  \end{enumerate}

More generally, the above items (1) - (2) hold if we add 
any number of operations, possibly of distinct arities, 
possibly satisfying distinct
properties chosen from (A1e) -   (D4).
 \end{theorem}

\begin{proof}
The proof goes exactly as in the proof of
\cite[Theorem 5.1]{sapimpap}.
Here we work with finite distributive lattices,
so the assumption (MID) applies, since  (MID) is equivalent
to distributivity in finite lattices.
\end{proof}
 
In particular, Theorem \ref{decid} applies when
$T$ is the theory of Boolean algebras,
the theory of distributive lattices or the theory
of semilattices. 
As in \cite{sapimpap}, the assumption that $T$ is finitely axiomatizable in (2)
can be weakened: it is enough to assume that 
it is decidable whether any given finite model 
is a model of $T$.

\begin{problem} \labbel{prob}
  \begin{enumerate}[(a)]    
\item 
Do Heyting algebras (not the dual!)
satisfy the hyperamalgamation property?
\item
Provide results similar to the results
proved in the present paper  by taking into
account further properties of $K$, 
or by relating the properties of two distinct operations,
say, asking that $K^nK_1^m(x)=K_1^pK^q(x)$,
for some $n,m,p,q \in \mathbb N$.  
\item
In particular, as in \cite{sapimpap}, consider
distinct operations required to satisfy a set of
comparability properties. Some cases
are straightforward, compare \cite{sapimpap},
but it is likely that the case dealing with
idempotent operations needs a few nontrivial 
details.
  \end{enumerate} 
 \end{problem}

\section{Further remarks} \labbel{fur}

\begin{remark} \labbel{nonimpl}
The hyperamalgamation property is actually a stronger  property
 than the superamalgamation property. 
For example, consider the class $\mathcal B_3$  of bounded lattices
(with constants ``top''  $t$ and ``minimum''  $m$)
of height $\leq 3$. 
Members of $\mathcal B_3$ are essentially
``sets of atoms'', together with two more elements   $t$ and $m$.
Thus, given a triple  $\mathbf A$, $\mathbf  B$, $\mathbf  C$ to be amalgamated,
an amalgamating structure $\mathbf  D$ can be obtained by just
considering the set of atoms of both $\mathbf A$ and $\mathbf  B$,
identifying the atoms belonging to $C$. The order is given by
the trivial relations: everything is greater than $m$ and smaller than $t$
and no more nontrivial order relations hold.  
The superamalgamation property holds straightforwardly,
precisely because the order is trivial; indeed, we never have nontrivial
inequalities of the form $a \leq b$. 

On the other hand, the hyperamalgamation property fails 
for $\mathcal B_3$.
Let $C= \{ m, t \} $,
let $\mathbf A$ have at lest two atoms $a$ and  $a_1$
and  $\mathbf B$ have at lest an atom   $b_1$.
In any strongly amalgamating structure,
$a \leq a_1 + b_1$, but $m$ is the only element 
of $\mathbf  C$ which is $\leq b_1$, so if we have
$a_2 \leq a_1 + m$, for some element $a_2$,   
then necessarily $a_2 \leq a_1$.
 Iterating, if there are elements as in the conclusion of
the hyperamalgamation property, we get
$b_2 \leq b_1$, $a_3 \leq a_2 \leq a_1$,\dots, 
$a \leq a_n \leq \dots \leq a_2 \leq a_1$.
But $a \leq a_1$ does not hold in $\mathbf A$,
so we have got a contradiction.  
 \end{remark}   

We now can show that the hyperamalgamation assumption 
is necessary in Proposition \ref{suithyp}.

\begin{example} \labbel{necess}
(a) Consider the class $\mathcal B_3$ as above and add
an additive operation $K$.
Let $C= \{ m, t \} $ with $Km=m$, $Kt=t$, $A= \{ m, a, t \} $
with $Ka=m$ and $B= \{ m, b, t \} $
with $Kb=b$. Then $K$ 
is additive and idempotent in $\mathbf A$, $\mathbf  B$ and $\mathbf  C$.
However, in any amalgamating structure belonging to  $\mathcal B_3$
we necessarily have $t=a+b$, but 
$Ka + Kb=m+b=b \neq t =Kt$, so that
$K$ cannot be additive.

In particular, 
the assumption that $\mathbf E^-$
is a hyperamalgamating structure is necessary 
in  Proposition \ref{suithyp}. In words,
``suitability'' is not necessarily preserved
by superamalgamating extensions: we do need 
hyperamalgamating extensions.  

(b) The above example treats the cases (D1) and (D3),
but can be easily modified in order to deal with the remaining cases.

Consider the class $\mathcal B_4^*$  of bounded lattices
of height $\leq 4$ and such that there is exactly 
one element $s$ strictly preceding $t$.
Again, 
members of $\mathcal B_4^*$ are essentially
``sets of atoms'', together with $3$  more elements   $t$, $s$  and $m$,
so that, as in Example \ref{nonimpl}, $\mathcal B_4^*$ has 
the superamalgamation property.

Now consider a language with also
an additive operation $K$,
let $C= \{ m, s, t \} $ with $Km=m$, $Ks=Kt=t$, $A= \{ m, a, s, t \} $
with $Ka=a$ and $B= \{ m, b, s, t \} $
with $Kb=b$. Then $K$ 
is a closure operation in $\mathbf A$, $\mathbf  B$ and $\mathbf  C$.
However, in any strongly amalgamating structure in $\mathcal B_4^*$,
$a+b=s$, so that $K(a+b)=Ks=t\neq s = a+b=Ka+Kb$,
hence we cannot have a strongly amalgamating structure in which  
$K$ is additive.
 \end{example}

\acknowledgement{We thank an anonymous reviewer of
\cite{sapimpap} for many interesting comments and
suggestions.}


\begin{thebibliography}{BD}    


\bibbitem{BD}
Balbes, R., Dwinger, P.,
 \emph{Distributive lattices},
University of Missouri Press, Columbia, MO
(1974).

\bibbitem{BHKK} 
Beyarslan, \"O.,  Hoffmann,  D. M., Kamensky,  M., Kowalski,
P.,
\emph{Model theory of fields with free operators in positive characteristic},
Trans. Amer. Math. Soc. \textbf{372}, 5991--6016  (2019). 


\bibbitem{BGR} Bruttomesso, R., Ghilardi, S., Ranise, S. 
\emph{Quantifier-free interpolation in combinations of equality 
interpolating theories}, ACM Trans. Comput. Log. \textbf{15}: Art. 5, 34 (2014).

\bibbitem{CHK} 
Chajda, I., Hala\v{s}, R., K\"{u}hr, J.  
\emph{Semilattice structures},
Research and Exposition in Mathematics
\textbf{30},
Heldermann Verlag, Lemgo (2007).


\bibbitem{E}  Ern\'e, M.,
 \emph{Closure},
in Mynard, F.,  Pearl E. (eds), 
\emph{Beyond topology},
  Contemp. Math.
\textbf{486},
Amer. Math. Soc., Providence, RI,
163--238
(2009).


\bibbitem{GM}
 Gabbay, D. M., Maksimova, L., 
\emph{Interpolation and definability. Modal and intuitionistic logics},
Oxford Logic Guides
\textbf{46}, The Clarendon Press, Oxford University Press, Oxford
(2005).

\bibbitem{GH} Gehrke, M.,   Harding, J.,
\emph{Bounded lattice expansions}, J. Algebra
 {\bf 238}, 345--371  (2001).


\bibbitem{GP} Guzy, N., Point,  F.,
 \emph{Topological differential fields},
 Ann. Pure Appl. Logic
 \textbf{161},
570--598 (2010).


\bibbitem{GG} 
Ghilardi, S.,  Gianola, A., 
\emph{Modularity results for interpolation, amalgamation and superamalgamation}, 
Ann. Pure Appl. Logic \textbf{169}, 731--754 (2018).


\bibbitem{G} Gr\"atzer, G., \emph{Lattice theory: foundation},
 Birkh\"auser/Springer Basel AG, Basel, 2011. 


\bibbitem{Ha}  Harzheim, E., \emph{Ordered sets},  
Advances in Mathematics \textbf{7},  New York, 2005.


\bibbitem{H} 
Hodges, W., \emph{Model theory},
 Encyclopedia of Mathematics and its Applications \textbf{42},
 Cambridge University Press, Cambridge, 1993. 

\bibitem{J} J{\'o}nsson, B.,
\emph{Universal relational systems},
Math. Scand. \textbf{4}, 193--208  (1956).   


\bibbitem{J1} B. J{\'o}nsson,
\emph{Extensions  of  relational  structures}, 
in \emph{Theory  of  Models}  (Proc.  Internat.  AP 
Sympos.  Berkeley,  1963),  North-Holland,  Amsterdam,   146--157  (1965)



\bibbitem
{KH}  Kihara, H., Ono, H.,
\emph{Interpolation properties, Beth definability
 properties and amalgamation properties for substructural logics},
 J. Logic Comput. \textbf{20},   823--875 (2010).


\bibbitem{KMPT} 
 Kiss, E. W., M\'arki, L., Pr\"ohle,  P.,  Tholen,  W.,
\emph{Categorical algebraic properties. 
A compendium on amalgamation, congruence extension, 
epimorphisms, residual smallness, and injectivity},
 Studia Sci. Math. Hungar. \textbf{18},  79--140  (1982). 

\bibitem{Ko}  
Koppelberg, S.,
\emph{Handbook of {B}oolean algebras. {V}ol. 1},
North-Holland Publishing Co., Amsterdam
(1989).


\bibbitem{ecca}  Lipparini, P.,
\emph{Existentially complete closure algebras},
Boll. Un. Mat. Ital. D (6)
\textbf{1},
 13--19 (1982).

\bibbitem{apuvar} Lipparini, P.,
\emph{Varieties defined by basic equations have the amalgamation
  property},
{Comm. Algebra} \textbf{52},
 4786--4805  (2024). 

\bibbitem{sapimpap} 
 Lipparini, P., \emph{Preservation of superamalgamation by expansions}, 
Fund. Math. {\bf 272},  205--237  (2026).

\arxiv{

\bibbitem{apu}  P. Lipparini, \emph{The strong amalgamation property into union},
 arXiv:2103.00563, 1--36 (2021).

}

\bibbitem{MT}  McKinsey, J. C. C., Tarski, A., \emph{The algebra of
   topology}. Ann. of Math.  \textbf{45} (1944), 141--191



\bibbitem{MMT}  Metcalfe, G.,  Montagna,  F.,  Tsinakis, C., 
\emph{Amalgamation and interpolation in ordered algebras}, J. Algebra 
\textbf{402},  21--82 (2014).


\bibbitem{Sc} Scowcroft, P.,
\emph{Existentially closed closure algebras},
Notre Dame J. Form. Log. \textbf{61},
623--661 (2020).


\bibbitem{Se} Servi, M.,
\emph{Sulla meno fine topologia ottenuta per estensione da una
              infratopologia generalizzata},
Univ. e Politec. Torino Rend. Sem. Mat.
\textbf{23},
237--248
(1963/64).


\end{thebibliography}
\end{document}